\documentclass[paper,11pt]{AAS} 

\usepackage{bm}
\usepackage{amsmath}
\usepackage{subfigure}
\usepackage[colorlinks=false, pdfstartview=FitV, linkcolor=red, citecolor=red,urlcolor=red]{hyperref}
\usepackage{overcite}
\usepackage{amssymb}
\usepackage{enumitem}
\usepackage{algorithm}
\usepackage{algpseudocode}
\usepackage{graphicx}
\usepackage{amsthm} 
\theoremstyle{definition}

\newtheorem{proposition}{Proposition}
\newtheorem{theorem}{Theorem}
\newtheorem{remark}{Remark}

\PaperNumber{25-056}

\begin{document}

\title{Adaptive Grid Refinement for Optimal Feedback Control}

\author{Alen E. Golpashin\thanks{PhD Candidate, Department of Aerospace Engineering, University of Illinois, Urbana, IL, 61801, USA.}
\ , 
Matthew T. Clark\thanks{PhD Student, Department of Aerospace Engineering, University of Illinois, Urbana, IL, 61801, USA.}
\ and
Bruce A. Conway\thanks{Professor Emeritus \& Research Professor, AAS Fellow, Department of Aerospace Engineering, University of Illinois, Urbana, IL, 61801, USA.}
}

\maketitle{}

\begin{abstract}
This paper introduces an algorithmic approach for adaptive grid refinement for solving the Hamilton-Jacobi-Bellman (HJB) partial differential equation (PDE) to obtain the optimal feedback control. Traditional grid generation methods often face challenges associated with state-space dimensionality, solution regularity, and large variations in the solution. Whereas fixed grid generation methods may fail to capture the intricate details of the optimal solution, our methodology addresses this challenge by dynamically refining the grid used to solve the HJB equation, thereby optimizing computational resource usage while targeting improved solution accuracy. The reinforcement learning (RL)-based approach presented in this study seeks to reduce the node population while constructing a grid with higher node density in regions with large discretization error.
\end{abstract}

\section{Introduction}
Grid-based computations offer several benefits, including enabling a structured representation of the state space, resolution of sharp gradients and nonlinearities in the solution, and compatibility with feedback control mechanisms. When combined with dynamic programming, these methods can achieve global optimality. Moreover, grid-based approaches enable parallelization for computational efficiency. These advantages make grid-based methods a powerful tool in applications of optimal control. When solving the HJB PDE globally, the grid structure significantly influences the error resulting from the spatial and temporal discretization of the continuous state space. The HJB equation is inherently complex and becomes even more difficult to solve when formulated for dynamics with high-dimensional state spaces. By discretizing this space into a finite grid, we can transform the continuous problem into a manageable set of algebraic equations that can be solved iteratively. Regular grid structures, in particular, allow efficient implementation of finite difference methods, commonly used to approximate derivatives in the HJB equation. This discretization process converts the continuous control problem into a form that can be handled by computational algorithms, making it possible to approximate the value function and derive optimal control policies. It is generally understood that a finer grid results in a smaller discretization error. However, this comes at the cost of increased computational resources. In particular, when solving higher-dimensional problems, the space complexity of regular grids grows exponentially with the number of dimensions of the system \cite{chilan2020}, leading to rapid exhaustion of available computational capacity.
To address a similar difficulty in high-dimensional integration, sparse grid methods were proposed as a promising approach to mitigate the curse of dimensionality. The foundational idea of sparse grids originates from Smolyak's 1963 work \cite{smolyak1963}, where he introduced a hierarchical approach for constructing sparse tensor products to approximate high-dimensional integrals. Building on this, Zenger later formalized and extended the use of sparse grids for numerical solutions of PDEs in 1991 \cite{zenger1991}. By avoiding the full tensor product structure, sparse grids reduce the number of nodes needed to approximate the infinite-dimensional solution of PDEs. These methods are particularly effective for moderately high-dimensional problems, where the computational cost of regular grids would be prohibitive. For instance, Kang and Wilcox demonstrated the effectiveness of a sparse grid characteristics method in solving HJB equations, showing that sparse grids can handle systems with up to six dimensions more efficiently than traditional full grids \cite{kang2017}. However, despite their efficiency, sparse grids have limitations. Their performance deteriorates as the dimensionality increases beyond a certain threshold, since the number of points still grows uncomfortably with the dimension. Besides their implementation complexity, sparse grids may struggle to capture localized features or steep gradients in the PDE solutions unless combined with adaptive refinement techniques.

Unstructured simplex meshes accommodate irregular domains and support piecewise-linear interpolation through barycentric coordinates. Their construction requires attention to element quality; for example, three-dimensional Delaunay triangulations can contain sliver tetrahedra that impair derivative interpolation \cite{persson2004}. Generating and updating the triangulation also adds computational cost. Aside from simplex meshes, quasi-Monte Carlo (QMC) sampling can be used to generate low-discrepancy unstructured grids \cite{chilan2020}. QMC grids are generated using Halton or Sobol sequences. These deterministic sequences are low-discrepancy which result in a more uniform coverage of the domain. For instance, consider an affine image $\Omega=\phi([0,1]^d)$ of the unit cube, with $x_i=\phi(q_i)$. Under the standard regularity and low-discrepancy assumptions for QMC integration, the quadrature error admits a bound of the form
\begin{equation*}
\left|\int_{\Omega} V(x)\,dx - \frac{|\Omega|}{n} \sum_{i=1}^n V(x_i)\right| \leq C \frac{(\log n)^d}{n},
\end{equation*}           
where $n$ denotes the number of sample points. For independent uniform Monte Carlo samples and finite-variance integrands, the root-mean-square quadrature error scales as $\mathcal{O}(n^{-1/2})$. For fixed $d$, the stated QMC error bound decays faster as $n\rightarrow\infty$. Due to their implementation simplicity and low time complexity, QMC sequences are ideal candidates for applications with a repetition component, i.e. training. However, these grids may need a sufficiently large number of sample point to compensate for the impact of $(\log n)^d$, especially in very high dimensions \cite{caflisch1998}.

The choice of interpolation method also plays a critical role in solving the HJB PDE. While QMC grids offer uniform coverage and low-discrepancy properties, their performance is highly dependent on how function values are approximated between nodes. For instance, while linear interpolation is computationally efficient and may suffice for smooth functions, it can introduce significant errors in regions with large variations. Higher-order interpolation methods, such as cubic splines, can provide more accurate approximations for sufficiently smooth functions but require careful consideration of node placement to avoid numerical artifacts. Therefore, refining the grid resolution and strategically placing nodes in regions where the value function changes most rapidly is advantageous for improving interpolation accuracy and reducing solution errors. In particular, the numerical solution of the HJB PDE has benefited from methods that leverage QMC sequences like Halton grids to reduce the space complexity of value iteration approximations in higher-dimensional problems. Chilan and Conway\cite{chilan2020} combine low-discrepancy QMC grids with Kriging interpolation to approximate HJB viscosity solutions on unstructured grids, avoiding the full tensor-product structure of regular grids. However, even QMC-based grids, which use quasi-deterministic sequences for node placement, may allocate similar node densities to regions that exhibit vastly different levels of complexity. As a result, regions with large discretization error or interpolation residuals may still suffer from inadequate node density, while simpler regions may be over-populated with nodes.

To address this limitation, this work implements an algorithm that iteratively adjusts grid nodes based on error estimation, concentrating computational efforts in regions where the solution exhibits significant complexity or rapid changes. By incorporating recent advancements in deep neural networks and reinforcement learning (RL), the algorithm generates feasible local actions that strategically places nodes based on learned criteria. Our primary objective is to capture critical features of the solution landscape with higher resolution, improving the accuracy of the optimal control policy representation. Here, we consider the grid node population as a function of the upper bounds of both time and space complexities and, by analogy, the computational burden. For node generation, we rely on QMC sampling to more uniformly  cover the domains of interest. 

Reinforcement learning has been utilized in several prior works for adaptive mesh and grid refinement; cf. References~\citenum{yang2023, freymuth2023, yang2023multi}. In particular, our adaptive approach builds upon insights from RL approach developed in Foucart et al.\cite{foucart2023} where local actions are confined to a bounded volume or cell of the grid. The reward function evaluates changes globally over the computational domain while guiding local node placement. The numerical results show increased node concentration in regions of rapid gradient variation, consistent with the interpolation-error bound. The agent is trained on an example two dimensional time-optimal control problem. The zero-shot learning ability of the algorithm is demonstrated using higher dimensional infinite-time horizon problems.  To demonstrate the efficacy of our approach, we then evaluate the performance of the trained algorithm in improving the solution process of the optimal control solvers. 

The paper is organized as follows: In the \nameref{sec:preliminaries} section, we review classical results on HJB approximations and convergence bounds to motivate our approach. In the \nameref{sec:formulation} section, we introduce the RL-based algorithm. \nameref{sec:baseline} outlines the baseline training procedures and results, while \nameref{sec:ablation} investigates the effects of various hyperparameters. Finally, in the \nameref{sec:experiments} section, we demonstrate the applicability of the method to aerospace benchmark problems. Code and data are available here\footnote{https://github.com/golpashin/AGR}.

\section{Preliminaries}\label{sec:preliminaries}
Consider the following optimal control problem
\begin{equation}
\begin{aligned}
& \underset{u(\cdot)}{\text{minimize}}
& & J(u) = \int_{t_0}^{\infty} r(x(t), u(t))\,\, e^{-\lambda (t - t_0)} \, dt  \\
& \text{subject to}
& & \dot{x}(t) = f(x(t), u(t)), \\
&&& x(t_0) = x_0, \\
&&& t \in [t_0, \infty)
\label{optimal_problem}
\end{aligned}
\end{equation}
where $U\subset\mathbb{R}^{m}$ is a nonempty compact set of admissible control values, $f:\mathbb{R}^{d}\times U\rightarrow\mathbb{R}^{d}$ describes the dynamics, and $r:\mathbb{R}^{d}\times U\rightarrow\mathbb{R}$ is the running cost. Both $f$ and $r$ are jointly continuous in $(x,u)$. The initial state is $x_0$, and $\lambda>0$ is the discount factor. The value function is defined by
\begin{equation}
V(x_0) \equiv \inf\limits_{u \in \mathcal{U}} J(x_0, u) 
\label{value_function}
\end{equation}
where $\mathcal{U}=\{u:[t_0,\infty)\rightarrow U:u\text{ is measurable}\}$ is the set of admissible control functions. We assume that $f$ and $r$ are uniformly bounded on $\mathbb{R}^{d}\times U$ and satisfy the following Lipschitz estimates, uniformly in $u\in U$
\begin{equation}
|f(x_1, u) - f(x_2, u)| \leq L_f |x_1 - x_2|,
\label{f_lipschitz}
\end{equation}
\begin{equation}
|r(x_1, u) - r(x_2, u)| \leq L_r |x_1 - x_2|,
\label{L_lipschitz}
\end{equation}
for any $u\in U$ and $x_1,x_2\in \mathbb{R}^{d}$.

\begin{proposition} (Lions~\cite{lions1985})
Assume relations \eqref{f_lipschitz} and \eqref{L_lipschitz} hold. Then the uniformly bounded function $V\in C(\mathbb{R}^d)$ is the viscosity solution of
\begin{equation}
\lambda V+\sup\limits_{u \in U}\Big\{ -f(x,u)\nabla V-r(x,u) \Big\}=0.
\label{hjb_viscosity}
\end{equation}  
\end{proposition} 
For a time step $0<h\leq 1/\lambda$, we introduce a consistent semidiscrete approximation of Eq. \eqref{hjb_viscosity}. Using an explicit Euler step for the dynamics and first-order approximations of the running-cost integral and discount factor gives the discrete dynamic programming equation
\begin{equation}\label{hjb_h}
V_h(x) = \min_{u \in U} \left\{ (1 - \lambda h) V_h(x + h f(x, u)) + h r(x, u) \right\}, \quad \text{for } x \in \mathbb{R}^d.
\end{equation}  
Under the stated boundedness and Lipschitz assumptions, for $\lambda>L_f$, the semidiscrete solution $V_h$ converges locally uniformly to the viscosity solution of Eq. \eqref{hjb_viscosity} as $h\rightarrow0$ \cite{dolcetta1983}. Moreover, the following error estimate holds
\begin{equation}
\|V - V_h\|_{\infty} \leq C h^{1/2},
\label{h_bound}
\end{equation}
where $C>0$ is a constant \cite{dolcetta1984}. Next, we  define the bounded set $A[0, 1]^d+D=\Omega\subset\mathbb{R}^d$, $A\in\mathbb{R}^{d\times d}$, $D\in\mathbb{R}^{d}$ as the  domain of computation and generate  the finite grid 
\begin{equation}
\mathcal{X} = \left\{ \phi(\mathbf{q}_i) \mid \mathbf{q}_i \in [0, 1]^d, \, i = 1, 2, \ldots, n \right\},
\label{grid}
\end{equation}
where $\phi(\mathbf{q})=A\mathbf{q}+D$ is an affine transformation and $A$ is nonsingular. For the spatial estimates below, let $\overline{\Omega}=\bigcup_{j\in\mathcal{J}}\tilde{S}_j$ be a triangulation whose vertex set is $\mathcal{X}=\{x_1,\ldots,x_n\}$. The space of continuous piecewise-linear functions is defined as \cite{falcone1987}
\begin{equation}\label{piecewise_linear}
\mathcal{W}
=\left\{w\in C(\overline{\Omega}) :w|_{\tilde{S}_j}\text{ is affine for every}\,j\in\mathcal{J}\right\}.
\end{equation}
We approximate the solution of Eq. \eqref{hjb_h} by a function in $\mathcal{W}$. For the theoretical estimates below, we assume that every evaluated footpoint $x_i+hf(x_i,u)$ belongs to $\overline{\Omega}$ for all $i=1,\ldots,n$ and $u\in U$. A footpoint need not be a grid node; its value is evaluated by interpolation on a containing simplex. The corresponding barycentric weights $\lambda_{ij}(u)$ are nonnegative, sum to one, and satisfy  
\begin{equation}
w(x_i+hf(x_i,u))= \sum_{j=1}^{n} \lambda_{ij}(u) w(x_j) \quad \text{and} \quad x_i+hf(x_i,u)= \sum_{j=1}^{n} \lambda_{ij}(u) x_j
\label{w_interp}
\end{equation}~
where $\sum_{j=1}^{n}\lambda_{ij}(u)=1$ for every $u\in U$. The fully discrete approximation $V_h^k\in\mathcal{W}$ is defined by
\begin{equation}\label{hjb_h_x}
V_h^k(x_i)
=\min_{u\in U}\left\{(1-\lambda h)\sum_{j=1}^{n}\lambda_{ij}(u)V_h^k(x_j)+h\,r(x_i,u)\right\},
\end{equation}
for $i=1,\ldots,n$. Existence and uniqueness follow from the contraction mapping argument \cite{falcone1987}. Define $F:\mathbb{R}^{n}\rightarrow\mathbb{R}^{n}$ by
\begin{equation}\label{fixed_point}
(F(W))_i=
\min_{u\in U}\left\{(1-\lambda h)\sum_{j=1}^{n}\lambda_{ij}(u)W_j+h\,r(x_i,u)\right\},
\qquad i=1,\ldots,n,
\end{equation}
where $W=(W_1,\ldots,W_n)^\top$ is the vector of nodal values. It can then be proved that Eq. \eqref{fixed_point} is a contraction mapping by showing
\begin{equation}\label{contraction}
\|F(W_1 )-F( W_2)\|_{\infty} \leq (1-\lambda h)\|W_1 - W_2\|_{\infty}, 
\end{equation}
which guarantees the existence of a unique nodal
fixed point and hence a unique solution
$V_h^k\in\mathcal{W}$ of Eq. \eqref{hjb_h_x}.
Let
\begin{equation}
k=\max_{j\in\mathcal{J}}
\operatorname{diam}(\tilde{S}_j)
\end{equation}
denote the largest simplex diameter. For $\lambda>L_f$, the spatial approximation error is bounded in terms of $k/h$ \cite{falcone1987}. Combining this spatial estimate with Eq. \eqref{h_bound} gives
\begin{equation}\label{bound_h_k}
\| V - V_h^k \|_\infty \leq C h^{1/2} + \frac{L_r}{\lambda (\lambda - L_f)} \frac{k}{h}
\end{equation}
for $\lambda>L_f$.

The following theorem establishes a bound on the difference between the semi-discrete HJB solution $V_h$ and its spatially discrete, piecewise-linear counterpart $V^k_h$. The aim is to show that this difference depends on the local interpolation errors. As a result, we shall propose adapting and refining the mesh where the local interpolation fails most significantly.    

\begin{theorem} 
Suppose $\lambda>L_{f}$. There is a constant $C>0$ (independent of $h$ and $k$) such that
\begin{equation}\label{interp_bound}
\|\,V_{h}-V_{h}^{k}\|_{\infty}\;\le\;
\frac{C}{\lambda\,h}
\bigl\|V_{h}-\mathcal{I}_{k}[V_{h}]\bigr\|_{\infty},
\end{equation}
where $\mathcal{I}_{k}[V_{h}]$ is the piecewise--linear interpolant of $V_{h}$ on $\mathcal{X}$.
\end{theorem}

\begin{proof} 
The proof of the theorem can be found in the Appendix.
\end{proof}
 
\begin{remark}\label{bramble_hilbert}
Suppose that $V_h\in C^2(\overline{\Omega})$ for each
admissible $h$ and that $\sup_h \|D^2V_h\|_{L^\infty(\Omega)}<\infty$. On a shape-regular family of triangulations, the Bramble--Hilbert lemma and the stability of piecewise-linear nodal interpolation yield\[
\|V_h-\mathcal{I}_k[V_h]\|_\infty
\le C k^2 \|D^2V_h\|_{L^\infty(\Omega)},
\]
where $C$ is independent of $h$ and $k$. Consequently, \eqref{interp_bound} yields
\[
\|V_h-V_h^k\|_\infty
\le \frac{C k^2}{\lambda h}
\|D^2V_h\|_{L^\infty(\Omega)}.
\]
For fixed $h$, this gives an $O(k^2/h)$ estimate as
$k\to0$.
\end{remark}

Therefore, we conclude that for fixed $h$, careful placement of the grid nodes $\{x_i\}$ can tighten the spatial discretization error bound by reducing the maximum local interpolation error.

For time-optimal control formulations, we consider an undiscounted problem with $r(x,u)\equiv1$ and take $t_0=0$. The optimal time-to-go function is defined as
\begin{equation}
T(x_{0})=\inf_{u\in{\mathcal{U}}} \Big\{ t\ge0: y_{x_0}(t,u)\in{\mathcal{T}} \Big\}=t_f^{*}(x_0) \label{optimal_time}
\end{equation}
where $t_f^{*}$ is the optimal terminal time, and $\mathcal{T}$ denotes the target set. Falcone\cite{falcone1994} showed that under some controllability conditions, \eqref{optimal_time} satisfies the boundary value problem
\begin{equation}
\begin{cases} \begin{aligned} \max_u \{-f(x,u).DT(x)\}=1 \end{aligned} & \text{in } \mathcal{R}\setminus\mathcal{T}\text{,} \\
T(x)=0 & \text{in } \mathcal{T}\text{,}  \\
\begin{aligned} \lim_{x\rightarrow  \varsigma^{\mathcal{}}} T(x)=+\infty \end{aligned} & \forall \varsigma^{\mathcal{}}\in\mathcal{\partial\mathcal{R}}\text{.} \\
\end{cases} \label{2bvp}
\end{equation}  
where $f:\mathbb{R}^{d}\times U\rightarrow\mathbb{R}^{d}$ is a vector field describing a dynamical system,  $x(t_f)$ and $x(t_0)$ are the terminal and initial states of the trajectory, respectively, and the Lagrangian is defined as $L(x(t), u(t), t)\equiv1$. 
Here, $\partial\mathcal{R}$ denotes the boundary of the reachability set, $\mathcal{R}$. We use the Kru\v{z}kov  change of variable

\begin{equation}
v(x)=\begin{cases} \begin{aligned} 1-e^{-T(x)} \end{aligned} & \text{if } T(x)<+\infty \\
1 & \text{if } T(x)=+\infty \\
\end{cases} \label{kruzkov}
\end{equation}
to obtain a bounded value function. Under the controllability and comparison hypotheses of Ref.~\citenum{falcone1994}, the minimum-time function is characterized by
\eqref{2bvp} in the appropriate nonnegative solution class.
The transformed value function satisfies
\begin{equation*}
v(x)+\max_{u\in U}\{-f(x,u)\cdot Dv(x)\}=1,
\qquad x\in\mathcal{R}\setminus\mathcal{T},
\end{equation*}
with $v=0$ on $\mathcal{T}$. Under the controllability and regularity assumptions of Ref.~\citenum{bardi1990}, the discrete minimum-time approximation admits convergence estimates on compact subsets of the reachable set. The convergence rate depends on those assumptions. A spatial approximation bound for the Isaacs equation, the counterpart of Eq. \ref{2bvp}, was derived in (Ref.~\citenum{soravia1998}).

\section{Problem Formulation}\label{sec:formulation}
In the previous section, we established that temporal and spatial discrete approximations depend on the time step $h$, grid size $k$, and the regularity of the dynamics and cost function. The error bound \eqref{interp_bound} shows that the approximation can be improved through localized grid modifications by removing nodes from regions with low interpolation error and increasing node density in regions where interpolation introduces the most uncertainty. However, improving approximations and balancing the computational load through manual node placement is cumbersome. For instance, from approximation \eqref{bound_h_k}, we can infer that one would need to place or remove nodes in a manner that accounts for the variation and regularity properties of $r$, $f$, and the solution which may be inhomogeneous in   $\Omega$, while also reducing $k$ in regions where the approximation is less sensitive. 

Instead in this work, for a given $h$, we consider the problem of adaptive node refinement in $\Omega$, where the objective is to iteratively adjust the node distribution to minimize approximation error between $V_h$ and $V_h^k$ while balancing computational cost. The grid-refinement environment is formulated as a Markov Decision Process (MDP). The agent receives a local observation $s$ of the full environment state and learns a policy $\pi(a|s)$ to add or remove nodes within a selected domain. Thus, for a given time step $h$, solution function on an initial grid, and an interpolant, the agent modifies the grid to lower the interpolation error over \(\Omega\), while being encouraged to reduce the node population \(n\). The agent's observation and action spaces are $\mathcal{S}\subseteq\mathbb{R}^3$ and $\mathcal{A}\subseteq\mathbb{R}^2$, respectively.

Borrowing from the methodology in Foucart et al.\cite{foucart2023}, we randomly and uniformly sample a point $y$ from $\operatorname{int}(\Omega)$ for the action to be applied to it. At each training step $\kappa$, the agent receives local observations $s_\kappa\ \in \mathcal{S}$ and applies local actions \(a_\kappa \in \mathcal{A}\) to the same sampled point $y$. Hence, aside from a global reward signal, the agent does not observe any information from other nodes. The following subsections detail the state, action, and reward components of this MDP.    

\subsection{Grid Initialization}
We initialize the computational domain $\Omega\subset\mathbb{R}^d$. Then the initial grid $\mathcal{X}_0$ is generated using QMC sampling. The Sobol sequence generated, $\phi:[0, 1]^d\rightarrow\overline{\Omega}$, carries out scrambling and skipping techniques to maintain uniformity and randomness in the initial grid. 
Subsequently at each step $\kappa$, one node is sampled from a uniform distribution $y \sim A \, \text{Unif}([0,1]^d) + D, \quad y \in \Omega\subset\mathbb{R}^d$. The initial condition $x(t_0)$ and the target state $x(t_f)$ are included (fixed) in $\mathcal{X}$, and boundary constraints are applied.

\subsection{State Representation}
The observation at step $\kappa$, $s_{\kappa} =(\rho,\varepsilon,d_{\partial})_{\kappa}\in \mathcal{S}$, is represented by:
\begin{itemize}[align=left, labelsep=2pt, leftmargin=*]
\item \textbf{Local Density Estimate:} The local node density $\rho(\mathbf{x})$ is computed using $\mathbf{k}$-nearest neighbor (kNN) analysis
\begin{equation}
\rho(y) = \frac{\mathbf{k}}{n V_d(R_\mathbf{k})},
\end{equation}
where $R_\mathbf{k}$ is the distance to the $\mathbf{k}$-th nearest neighbor, $n$ is the total number of nodes as before, and $V_d(R_\mathbf{k})$ is the volume of a $d$-dimensional ball with radius $R_\mathbf{k}$. We have chosen to set $\mathbf{k} = \operatorname{round}(\sqrt{n})$ to approach a tradeoff between bias and variance.
    
\item \textbf{Local HJB Residual:} The local HJB error $\varepsilon(y)$ at a point $y$ is computed as
\begin{equation}
\varepsilon(y) = \lambda V(y) - f(y, u(y)) \cdot \nabla^{\mathcal{W}} [V(y)] - r(y, u(y)),
\label{HJB_residual}
\end{equation}
where $V(y)$ and $u(y)$ are the interpolated value function and control at the sampled point $y$, respectively, when $y \notin \mathcal{X}$. The gradient is numerically computed using a finite difference scheme; the $\nabla^{\mathcal{W}}$ notation emphasizes that the operator assumes differentiability. The viscosity solution may be non-differentiable. For piecewise-linear interpolation, the interpolant is differentiable within each simplex and almost everywhere in the domain. Across simplex interfaces, finite differences provide a numerical residual indicator rather than a classical pointwise derivative.

We use the magnitude of the residual $\varepsilon$ as an empirical indicator of local refinement needs. Because the control is interpolated at the evaluation point, this quantity measures consistency with that control. Reducing this indicator motivates adaptive node placement.

\item \textbf{Normalized Boundary Proximity Ratio:}
The Normalized Boundary Proximity Ratio, $d_{\partial}\in\mathbb{R}$, is a centering measure that quantifies the relative positioning of a sampled point $y$ within the domain $\Omega = A[0,1]^d + D$. We define this feature as the ratio of the minimum normalized distance to the boundary to the maximum normalized distance
\begin{equation}
d_{\partial}(y)=
\frac{\displaystyle\min_{1\le i\le d}\min\{z_i,1-z_i\}}
     {\displaystyle\max_{1\le i\le d}\max\{z_i,1-z_i\}},
\end{equation}
where $z=A^{-1}(y-D)$. By operating in the normalized space, this measure is invariant to the size or aspect ratio of $\Omega$. A ratio close to 1 indicates proximity to the center of the
normalized domain, while a ratio close to 0 indicates
proximity to its boundary.\end{itemize}

\subsection{Action Space}
At step $\kappa$, the agent selects an action $a_\kappa=(R, P)_\kappa \in \mathcal{A}$, where:
\begin{itemize}[align=left, labelsep=2pt, leftmargin=*]
\item $R \in [0, \mathbf{max_{r}}]$ scales the radius of a region around a sampled point in the domain where nodes are added or removed.
In this study, we have fixed $\mathbf{max_{r}}=0.45$ to avoid large local domains compared to the size of $\Omega$.    
\item $P \in [-1, 1]$ specifies the percentage change in node count within the $d$-ball of radius $R. R_\Omega$. 

Here, $R_\Omega$ denotes the radius of a $d$-dimensional ball that has the same volume as the defined domain. To make the action space adaptable for domains of varying sizes and shapes, we scale the output action $R$ by $R_{\Omega}$. This quantity provides a reference scale for when dimension $d$ changes.
In particular, for a given domain $\Omega = A[0, 1]^d + D$, its volume is
\begin{equation*}
V_{\Omega} = \lvert\det(A)\rvert,
\end{equation*}
where $\lvert\det(A)\rvert$ represents the scaling applied by the affine transformation. For a $d$-dimensional ball with radius $R$, the volume is given by
\begin{equation*}
V_d(R) = \frac{\pi^{d/2}}{\Gamma\left(\frac{d}{2} + 1\right)} R^d.
\end{equation*}

To determine $R_{\Omega}$, we solve for the radius of the $d$-ball with volume equal to $V_{\text{domain}}$.\ We have
\begin{equation}
R_{\Omega} = \left( \frac{V_{\Omega} \cdot \Gamma\left(\frac{d}{2} + 1\right)}{\pi^{d/2}} \right)^{\frac{1}{d}}.
\end{equation}
By normalizing the physical radius $\bar R$ using $R_{\Omega}$, we obtain the dimensionless action $R=\bar R/R_{\Omega}$, allowing the same action range to be used under uniform changes in domain size.

Let us denote the true radius by $\bar{R}=R.R_{\Omega}$. The process of mapping QMC nodes within a circular domain of radius $\bar{R}$ around a sampled point $y$ is implemented such that it avoids node duplication. Given the initial nodes $\{x_{i}\} \subset \mathbb{R}^d$, the number of new nodes, $\Delta n^{+}$, is calculated as a function of the percentage $P$ and the count of existing nodes within the radius $\bar{R}$. This is computed as
\begin{equation*}
\Delta n^{+}=\operatorname{round}(|P|n_{\bar R}),
\qquad P>0,
\qquad
n_{\bar R}=\sum_{\|x_i-y\|\le\bar R}1.
\end{equation*}

Each new node is generated from a Sobol point $\mathbf{q}=(q_1,\ldots,q_d)\in[0,1]^d$. For $d\ge2$, the radial and angular coordinates within the ball of radius $\bar R$ are defined by
\begin{equation*}
r_{\text{ball}}=\bar R q_d^{1/d},
\qquad
\theta_i=F_{d,i}^{-1}(q_i), \quad i=1,\ldots,d-2,
\qquad
\theta_{d-1}=2\pi q_{d-1},
\end{equation*}
where
\begin{equation*}
F_{d,i}(\theta)=
\frac{\displaystyle\int_0^\theta
      \sin^{d-i-1}(s)\,ds}
     {\displaystyle\int_0^\pi
      \sin^{d-i-1}(s)\,ds},
\qquad 0\le\theta\le\pi.
\end{equation*}
This transformation maps uniform measure on $[0,1]^d$ to normalized volume measure on the $d$-dimensional ball. For $d=2$, the intermediate polar angles are absent, and the mapping reduces to $r_{\text{ball}}=\bar R\sqrt{q_2}$ and $\theta_1=2\pi q_1$. The radial and angular coordinates are then converted into Cartesian coordinates relative to $y$ as
\begin{equation*}
x_p = y_p + r_{\text{ball}} \left( \prod_{j=1}^{p-1} \sin(\theta_j) \right) \cos(\theta_p), \quad \text{for } p = 1, \dots, d-1,
\end{equation*}
and for the last coordinate
\begin{equation*}
x_d = y_d + r_{\text{ball}} \prod_{j=1}^{d-1} \sin(\theta_j).
\end{equation*}
where $x=(x_1, x_2,\dots, x_d)$ denotes the coordinates of newly added node at step $\kappa$. On the other hand, the agent removes a specified percentage $P$ of nodes from a $d$-dimensional spherical region of radius $\bar{R}$, centered around a sampled point $y$. The removal process excludes fixed nodes, e.g., initial and terminal states.
First, we identify the nonfixed nodes within distance $\bar R$ of $y$. The number of nodes to remove is calculated by rounding to the nearest integer.\footnote{Alternatively, the ceiling rule $\Delta n^{-}=\lceil |P|n_{\bar R}\rceil$ ensures that any nonzero removal fraction requests at least one node whenever the local removal pool is nonempty.}
\begin{equation*}
\Delta n^{-}=\operatorname{round}(|P|n_{\bar R}),
\qquad
n_{\bar R}=\sum_{\substack{\|x_i-y\|\le\bar R\\
x_i\text{ nonfixed}}}1.
\end{equation*}
where $P<0$ is the removal percentage now. If the rounded node count for the selected action is zero, the grid remains unchanged at step $\kappa$. This can occur for sufficiently small $|P|$ or when the corresponding local node set is empty. Figure \ref{action_circ} demonstrates the node addition mechanism for a domain $\Omega$ with $d=2$ and a hypothetical grid $\mathcal{X}$ with $\bar{R}=2$.    
\end{itemize}

\subsection{Reward Function}
The computation of reward function depends on factors such as global node count, global error, and agent's choice of action. In particular, we  compute the  global error using a new QMC sampled set of points in $\Omega$ at every step $k$.  We compute the following $L_2$ norm of the error residual
\begin{equation}
\mathcal{E}\equiv\| \varepsilon \|_{L_2} = \sqrt{\int_{\Omega} \varepsilon(\mathbf{x})^2 d\mathbf{x}}\approx\sqrt{\frac{|\Omega|}{N}\sum_{i=1}^{N}\varepsilon(\mathbf{x}_i)^2},
\end{equation} 
where $\{\mathbf{x}_i\}\in\Omega$, $i=1, 2, \dots,N$ are the new QMC-sampled query points,  and $\varepsilon$ is defined in Eq. \eqref{HJB_residual}. The query points are resampled at every step $\kappa$. Below, we introduce the reward components
\begin{itemize}[align=left, labelsep=2pt, leftmargin=*]
\item \textbf{Error-Based Reward:} The normalized reduction in error:
\begin{equation}
R_E = \max\left\{-1,\,1-\frac{\mathcal{E}_{\kappa}}{\mathcal{E}_{0}}\right\}.
\end{equation}
\item \textbf{Node-Based Reward:} The normalized reduction in node count:
\begin{equation}
R_N = \max\left\{-1,\,1-\frac{n_{\kappa}}{n_{0}}\right\}.
\end{equation}
\item \textbf{Clipping Penalty:} Penalties are applied for actions outside of predefined bounds using:

\begin{equation*}
P_a =
-\sum_{i=1}^{2}
\min\left\{
1,\,
\frac{
    \max\{a_{\min,i}-a_i,0\}
    +
    \max\{a_i-a_{\max,i},0\}
}{a_{\max,i}-a_{\min,i}}
\right\}.
\end{equation*}

\end{itemize}
where $\mathcal{E}_{\kappa}$ and $n_{\kappa}$ are the global integrated error and current node population at step $\kappa$. Also, $\mathcal{E}_{0}$ and $n_{0}$ denote the reference integrated error and node count on the initial grid $\mathcal{X}$ at step $\kappa=0$. The total step reward is computed as
\begin{equation}
    \mathcal{R}=R_E + R_N + P_{a},
\end{equation}
where $P_a$ represents the penalty. If $n_{\kappa}\notin[n_{\min},n_{\max}]$, the episode terminates and the step reward is instead $-\kappa_{\max}$.

\subsection{Learning Framework}
To train the agent for adaptive grid refinement, we employ the Proximal Policy Optimization (PPO) algorithm\cite{schulman2017}. This  reinforcement learning technique balances exploration and exploitation while ensuring stable policy updates. PPO optimizes a policy \( \pi_{\theta}(a|s) \) and a value function \( V_{\phi}(s) \) by iteratively improving the policy based on experience tuples \( (s_\kappa, a_\kappa, r_\kappa, s_{\kappa+1}) \). The actor (policy) and critic (value function) networks are represented as
\begin{equation*}
a=
\operatorname{scale}_{R,P}(f_{\mu}(s))+\epsilon\odot f_{\sigma}(s),
\qquad
V_{\phi}(s)
=
f_L\circ f_{L-1}\circ\cdots\circ f_1(s).
\end{equation*}
where $a$ is sampled from $\pi_{\theta}(\cdot\mid s)$ before clipping, $\epsilon \sim \mathcal{N}(0,I)$ introduces Gaussian exploration noise, $\odot$ denotes componentwise multiplication, and $f_i$ represents the layers of the neural network for the critic. Before clipping, the policy samples a raw action from the Gaussian distribution
\begin{equation*}
\pi_{\theta}(\cdot\mid s)
=
\mathcal{N}(\mu_{\theta}(s),\Sigma_{\theta}(s)).
\end{equation*}
where $\mu_{\theta}(s)$ is the mean vector, and $\Sigma_{\theta}(s)$ is a diagonal covariance matrix. PPO is particularly well-suited for this application because it provides an approach to policy updates that includes the use of clipped objective function. The clipped objective penalizes updates that deviate significantly from the current policy, ensuring that policy changes remain stable and incremental. This approach helps prevent overfitting to specific actions in a single problem, making it particularly useful in applications such as adaptive grid refinement, where the agent must generalize across a large set of state-action pairs in varying environments. Moreover, PPO's ability to handle continuous action spaces makes it particularly useful in noisy continuous problems such as grid refinement. Furthermore, the stability and scalability of PPO make it suitable for applications of adaptive grid refinement to higher dimensional problems.
 
Our neural network setup for PPO features an actor network with two standard paths: one for the mean and another for the standard deviation of the Gaussian action distribution. The actor has a shared trunk with layer sizes 64, 64, and 32, followed by separate mean and standard-deviation branches, each with a 32-neuron hidden layer and a two-neuron output layer. ReLU activations are used in the hidden branches, with Tanh and Softplus activations for the mean and standard-deviation outputs, respectively. The critic network $V_{\phi}(s)$ is composed of two layers with 64 neurons each and ReLU activations. To ensure stability during training, layer normalization is applied in all networks. Algorithm \ref{agr_alg} summarizes the environment interactions within a refinement episode. PPO updates the actor and critic using the collected transitions. 
\begin{figure}[htbp]
\centering
\includegraphics[width=2.8in]{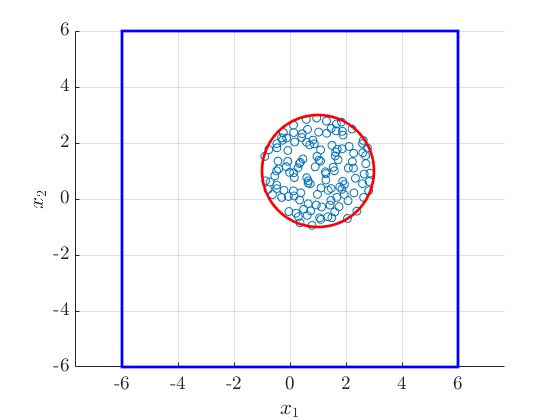}
\caption{Demonstration of Node Addition Mechanism for a Hypothetical Grid $\mathcal{X}$ within Domain $\Omega=[-6,6]^2$}
\label{action_circ}
\end{figure}

\begin{algorithm}[htbp]
\caption{Adaptive Grid Refinement Episode Used in PPO Training}
\label{agr_alg}
\begin{algorithmic}[1]
\Require Domain $\Omega$; initial data $(\mathcal{X}_0,V_0,u_0)$ with designated fixed nodes; reference residual norm $\mathcal{E}_0>0$; policy $\pi_\theta$; action bounds $a_{\min},a_{\max}$; node limits $n_{\min},n_{\max}$; maximum episode steps $\kappa_{\max}\ge1$
\Ensure Episode transitions for PPO and final grid data $(\mathcal{X},V,u)$

\State $\mathcal{X}\gets\mathcal{X}_0$, $V\gets V_0$, $u\gets u_0$
\State $n_0\gets|\mathcal{X}_0|$, $n\gets n_0$, $\kappa\gets0$, $\text{done}\gets\text{False}$
\State Initialize the interpolants of $V$ and $u$
\State Compute the equivalent-volume radius $R_\Omega$
\State Sample $y_0$ uniformly from $\Omega$ and compute its local observation $s_0$

\While{$\kappa<\kappa_{\max}\text{ and }\neg\text{done}$}
\State \textbf{Select and prepare the action:}
\State Sample $a_\kappa\sim\pi_\theta(\cdot\mid s_\kappa)$
\State $(R,P)\gets\operatorname{clip}(a_\kappa,a_{\min},a_{\max})$
\State $\bar R\gets R R_\Omega$

\State \textbf{Modify the grid:}
\If{$P>0\text{ and }\bar R>0$}
\State Count all nodes within distance $\bar R$ of $y_\kappa$ as $n_{\bar R}$
\State $\Delta n^+\gets\operatorname{round}(P n_{\bar R})$
\State Add $\Delta n^+$ QMC nodes within distance $\bar R$ of $y_\kappa$, subject to domain and duplicate checks
\ElsIf{$P<0$}
\State Count nonfixed nodes within distance $\bar R$ of $y_\kappa$ as $n_{\bar R}$
\State $\Delta n^-\gets\operatorname{round}(|P|n_{\bar R})$
\State Remove $\Delta n^-$ nodes sampled uniformly without replacement from this removal pool
\EndIf

\State Interpolate $V,u$ at added nodes and discard data for removed nodes
\State Rebuild their interpolants if the grid changed
\State $n\gets|\mathcal{X}|$

\State \textbf{Evaluate the reward and termination conditions:}
\State Estimate $\mathcal{E}_{\kappa+1}$ using newly sampled QMC query points in $\Omega$
\State Compute the clipped reward components $R_E,R_N$ using $\mathcal{E}_{\kappa+1},n$ and the references $\mathcal{E}_0,n_0$
\State Compute $P_a$ using the unclipped action $a_\kappa$
\State $\mathcal{R}_\kappa\gets R_E+R_N+P_a$
\State $\text{done}\gets(\kappa+1\ge\kappa_{\max}\text{ or }n<n_{\min}\text{ or }n>n_{\max})$
\If{$n<n_{\min}\text{ or }n>n_{\max}$}
\State $\mathcal{R}_\kappa\gets-\kappa_{\max}$
\EndIf

\State Sample $y_{\kappa+1}$ uniformly from $\Omega$ and compute its local observation $s_{\kappa+1}$
\State Supply $(s_\kappa,a_\kappa,\mathcal{R}_\kappa,s_{\kappa+1},\text{done})$ to the PPO trainer
\State $\kappa\gets\kappa+1$
\EndWhile

\State Compute final episode statistics
\State \Return $(\mathcal{X},V,u)$
\end{algorithmic}
\end{algorithm}

\section{Policy Optimization for the Baseline}\label{sec:baseline}
We consider the following 2-dimensional time-optimal control problem\cite{cristiani2010} to supply the training data. The two state variables are \(x_1\) and \(x_2\), and the control input \(u\) is a scalar. The system dynamics are defined as
\begin{equation*}\text{(P1)} \quad
\begin{aligned}
\dot{x}_1 &= c(x_1, x_2) \cos(u) \\
\dot{x}_2 &= c(x_1, x_2) \sin(u)
\end{aligned}
\quad
c(x_1, x_2) =
\begin{cases} 
1 & \text{if } x_2 \leq 1 \\
(x_2 - 1)^2 + 1 & \text{if } x_2 > 1.
\end{cases}
\end{equation*}
The goal is to compute \(T(x_0)\), the minimum time to reach the terminal state, given the initial condition \(x(t_0) = (-2.5, 0)\) and the terminal condition \(x(t_f) = (3, 0)\) in $\Omega=[-6, 6]^2$. The initial grid supplied to the agent at \(\kappa = 0\) is illustrated in Figure~\ref{initial_grid}.
The time-step used was $h=0.01$.
\begin{figure}[htbp]
 \centering
\subfigure[Initial QMC-generated Grid with 1500 nodes at $\kappa=0$. ]{
\includegraphics[width=0.47\textwidth]{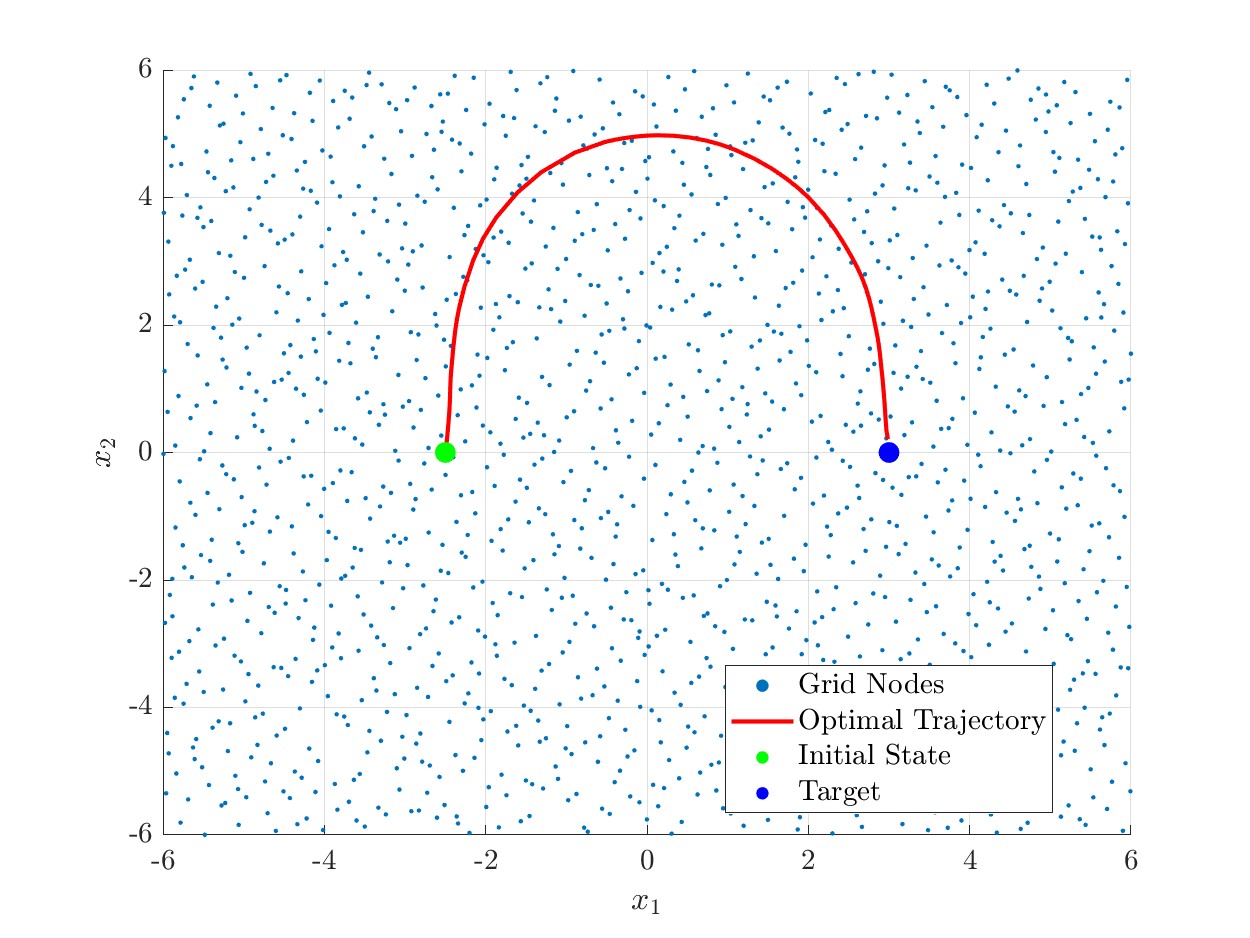}
\label{initial_grid}}
\hspace{0.02\textwidth}
\subfigure[Final Grid After Adaptive Refinement.]{
\includegraphics[width=0.47\textwidth]{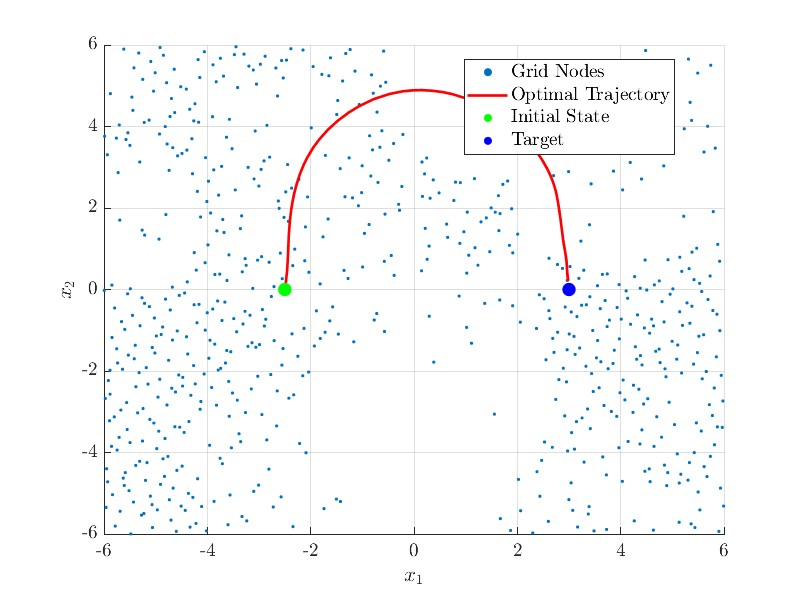}
\label{final_grid}}
\caption{Comparison of the Initial and Final Grids with Their Respective State Trajectories.}
\label{comparison_grids}
\end{figure}

After 200,000 training episodes, as detailed in Algorithm \ref{agr_alg}, the agent achieves modified grids (i.e. Figure~\ref{final_grid}) with learning curves provided in Figures~\ref{baseline_reward} and~\ref{low_min}. The former is the baseline agent and is trained with bounds $[n_{\text{min}}, n_{\text{max}}]=[450, 2250]$, and the latter is with  $[n_{\text{min}}, n_{\text{max}}]=[10, 2250]$.
For the test episode associated with figures \ref{comparison_grids}--\ref{surface_Hessian} here, the cumulative reward breakdown was $\sum_{\kappa}R_E=8.4013$ and $\sum_{\kappa} R_N=9.5832$, with the estimated integrated residual norm lowered from $6.8150$ to $4.8719$ and node population lowered from $n_0=1502$ nodes to $n_{\kappa_{\text{max}}}=666$. Figure \ref{comparison_grids} compares the state trajectories obtained on the initial and refined grids. In particular, in Figure~\ref{initial_hessian}, it is noteworthy that the agent increases the node concentration in regions with high variability and along lines of non-differentiability, consistent with the explanation provided in Remark~\ref{bramble_hilbert}. This is while the agent removes or relocates nodes from regions with less variability and solution details, as seen in the surface triangulation (Figure \ref{surface_tri}). Similarly, Figures \ref{initial_error} and \ref{final_error} suggest that the agent was able to lower the HJB residual magnitude over the domain $\Omega$, where the darker blue contour signifies a lower error.      
\begin{figure}[htbp]
\centering
\subfigure[Residual Contour on the Initial QMC-generated Grid]{
\includegraphics[width=0.47\textwidth]{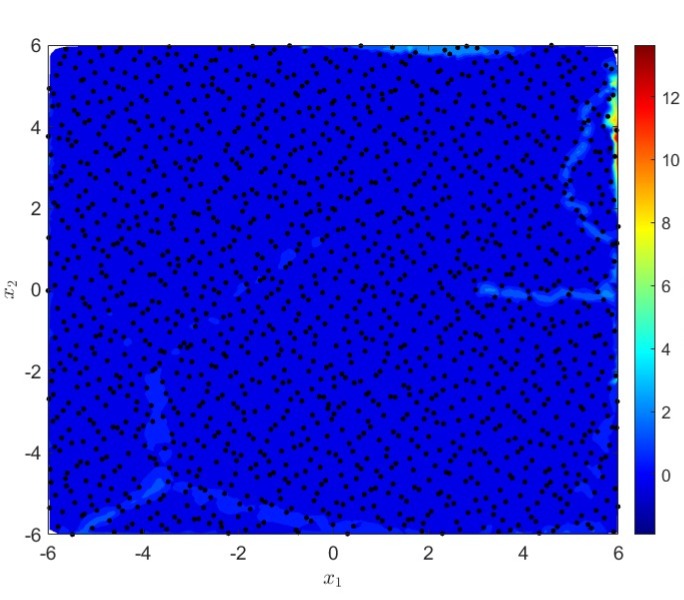}
\label{initial_error}}
\hspace{0\textwidth}
\subfigure[Final Residual Contour on the Modified Grid]{
\includegraphics[width=0.47\textwidth]{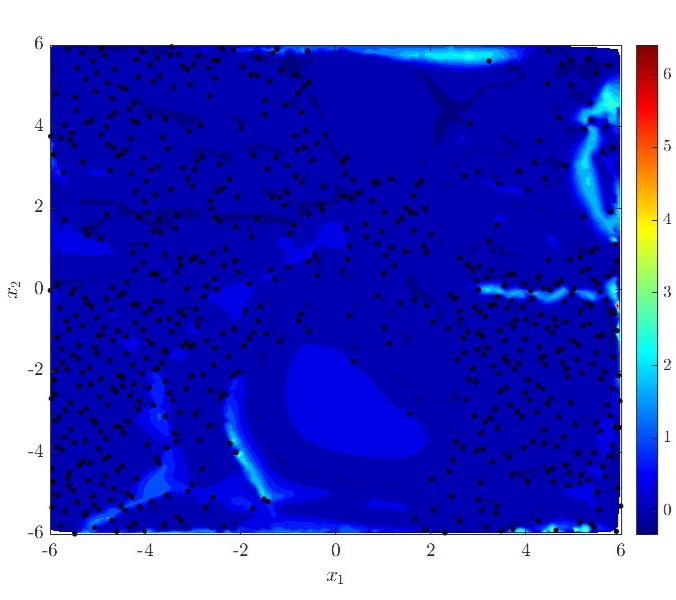}
\label{final_error}}
\caption{Function Value Contour of the Residual Function $\varepsilon$}
\label{contour_comparison}
\end{figure}
   
\begin{table}
\caption{Hyperparameters for Baseline PPO Agent}
\label{table_ppo_hyperparams}
\centering
\begin{tabular}{c | l | c}
\hline
\textbf{Parameter} & \textbf{Description} & \textbf{Value} \\
\hline
$\gamma$           & Discount factor for future rewards & 0.99 \\
$\lambda$          & GAE (Generalized Advantage Estimation) factor & 0.95 \\
$\epsilon$         & Clipping factor for PPO objective & 0.2 \\
$c_2$              & Entropy loss weight & 0.05 \\
$M$                & Mini-batch size & 512 \\
$K$                & Number of epochs per update & 3 \\
$\alpha_{\pi}$     & Actor network learning rate & $2 \times 10^{-5}$ \\
$\alpha_V$         & Critic network learning rate & $1 \times 10^{-3}$ \\
$\kappa_{\text{max}}$ & Number of steps per training episode & $32$
\\
\hline
\end{tabular}
\end{table}

\begin{figure}
\centering
\subfigure[Learning Curve of the Baseline, Trained on the Converged Value Function $v(x)= 1-e^{-T(x)}$ of Problem 1]{
\includegraphics[width=0.46\textwidth]{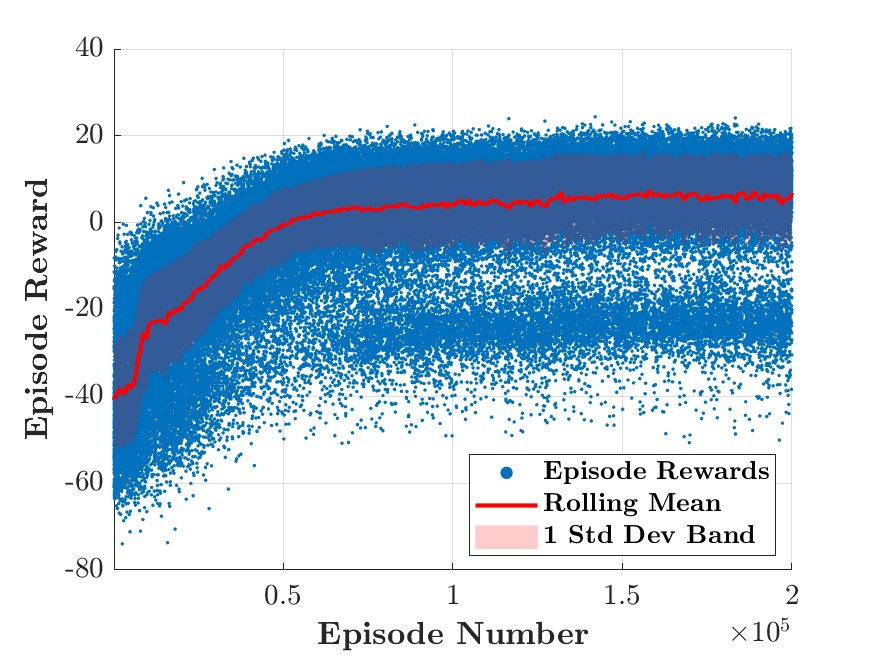}
\label{baseline_reward}}
\hspace{0.02\textwidth}
\subfigure[Learning Curve of the Agent Trained with 10 Minimum and 2250 Maximum Node Limits. Also Trained on Data from Problem 1]{
\includegraphics[width=0.46\textwidth]{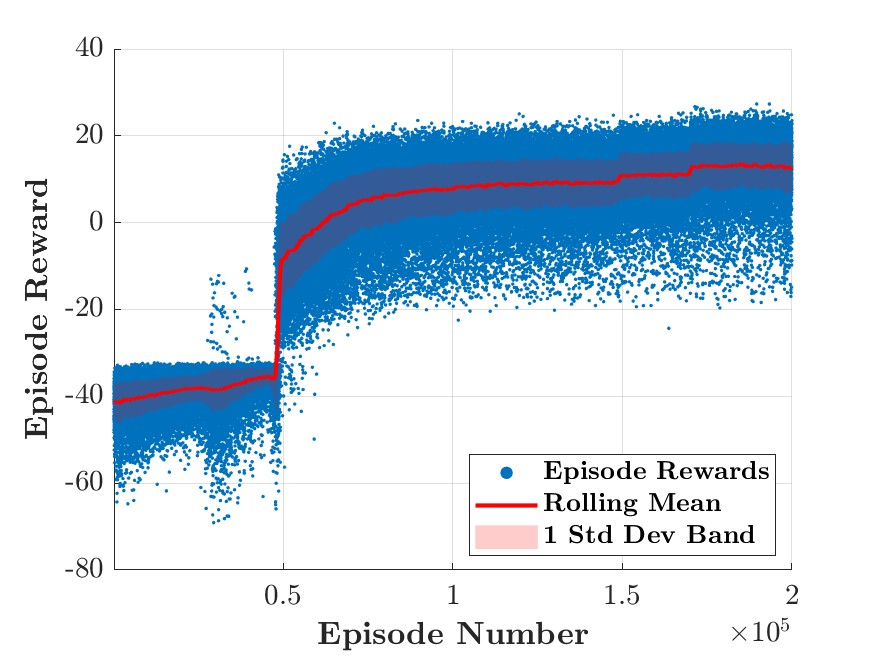}
\label{low_min}}
\caption{Comparison of Learning Curves between the Baseline and the Case with Broader Node Population Limits}
\label{fig:comparison_learning_curves}
\end{figure}
\begin{figure}
\centering
\subfigure[Solution Surface $V=T_h=-\log(1 - v_{h})$ and the Triangulation on the Modified Final Grid]{
\includegraphics[width=0.46\textwidth]{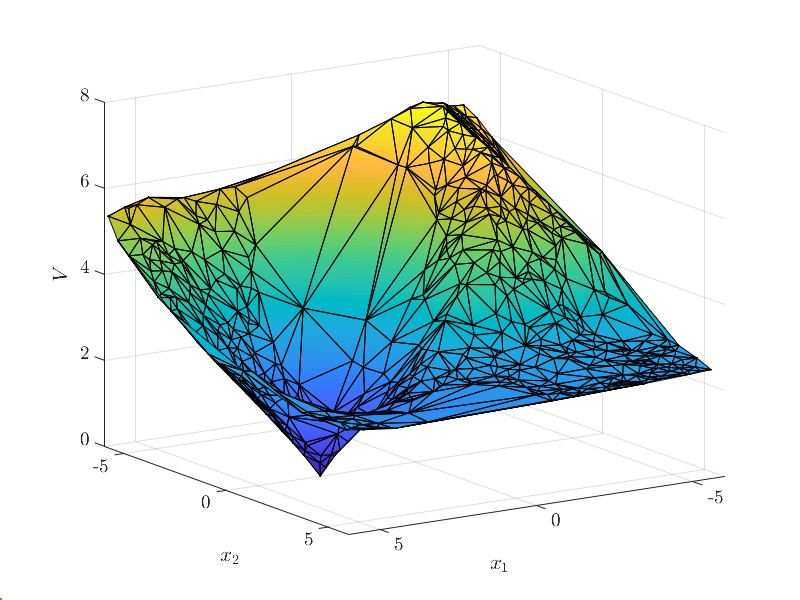}
\label{surface_tri}}
\hspace{0.02\textwidth} 
\subfigure[Frobenius Norm of the Hessian of $T$, Computed as $\log (1 + ||\nabla^2 T||)$ on the Initial Grid]{
\includegraphics[width=0.46\textwidth]{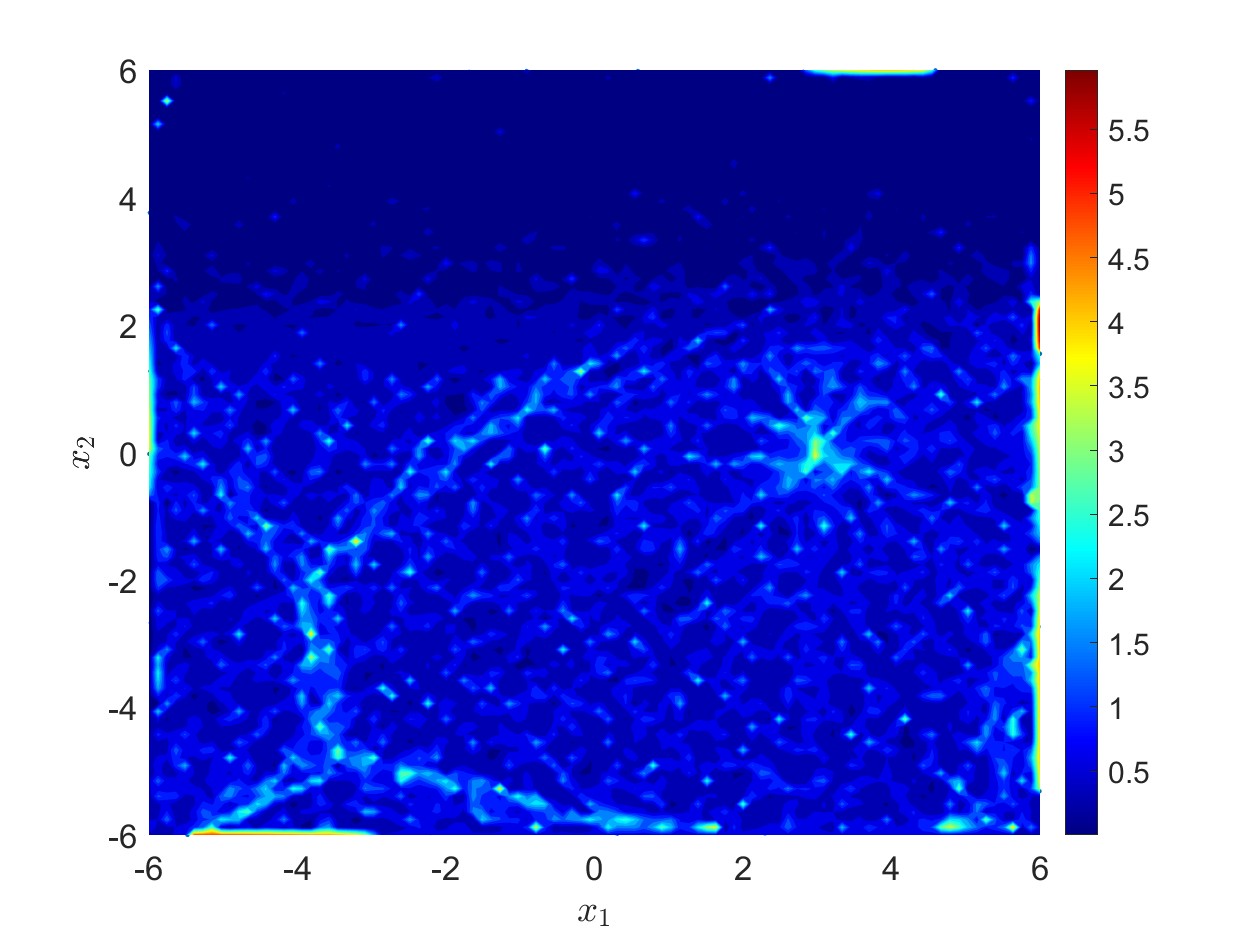}
\label{initial_hessian}}
\caption{Depiction of Solution Surface on the Final Grid and the Solution Variability on the Initial Grid}
\label{surface_Hessian}
\end{figure}

\begin{figure}[H]
\centering\includegraphics[width=3.7in]{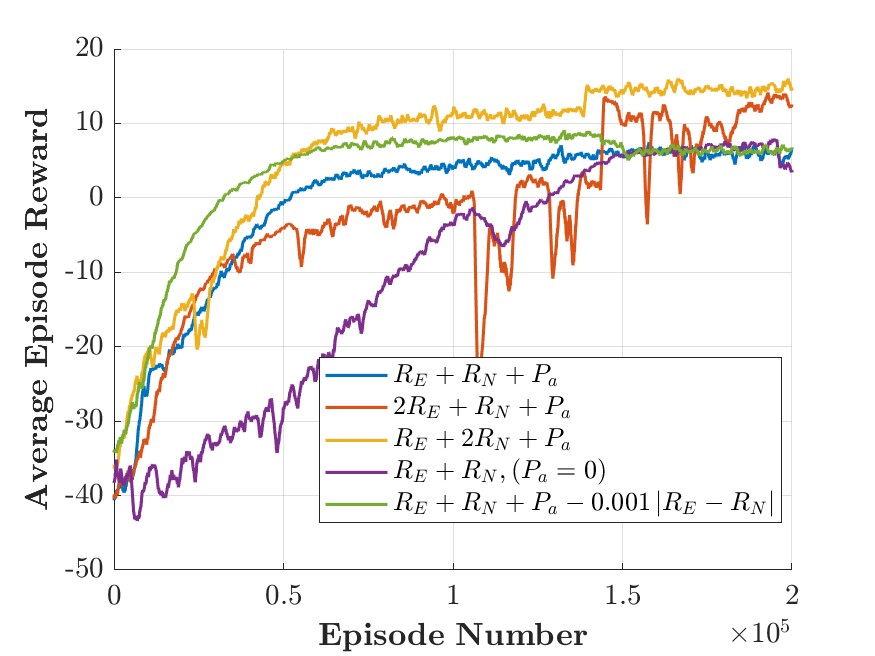}
\caption{Reward Function Variation, Trained with Baseline Hyperparameters}
\label{reward_variation}
\end{figure}

Table~\ref{table_ppo_hyperparams} shows the hyperparameters of the baseline case for training the agent on Problem 1. In Figure \ref{reward_variation}, we have tested five different reward functions. The baseline refers to the case described earlier with $\mathcal{R}=R_E+R_N+P_a$. The cases favoring error reduction and favoring node reduction each have a step reward defined as $\mathcal{R}=2R_E + R_N + P_{a}$ and $\mathcal{R}=R_E + 2R_N + P_{a}$, respectively. We also  tested a reward function with a balancing term to force the agent to create a more optimal trade-off between approximation accuracy and computational  burden. The reward for this case  is given by $\mathcal{R}=R_E + R_N + P_{a}-0.001|R_E-R_N|$. Finally, the reward associated with no action penalty is similar to the baseline case, but with $P_a=0$ for all $\kappa$. From Figure \ref{reward_variation}, we find that the case with no out-of-bounds action penalty struggles with learning efficiency. Similarly, when agent was forced to prioritize error reduction, the learning became unstable which signals a need for further hyperparameter tuning. In general, the agent found it easier to prioritize node reduction as seen for the case with $\mathcal{R}=R_E + 2R_N + P_{a}$. This observation may become obvious when considering the complexity involved when refining node geometries to reduce approximation uncertainty.

\section{Ablation Study}\label{sec:ablation}
Figures \ref{clip_ablation}--\ref{batch_ablation} describe the training results obtained for various hyperparameters, compared against the baseline agent (Table~\ref{table_ppo_hyperparams}). In particular, Figure \ref{entropy_ablation} demonstrates the impact of different entropy coefficients, \(c_2\), on learning stability and exploration. Generally, smaller entropy coefficient, i.e. \(c_2 = 0.01\), result in a more stable and effective training in the environment. Surprisingly, the $c_2=0.01$ outperforms the baseline case. 
Figure \ref{rates_ablation} shows the influence of the learning rates on the actor \(\alpha_\pi\) and critic \(\alpha_V\). We see that higher learning rates for the actor, i.e. \(2 \times 10^{-4}\), lead to faster initial improvement but introduce instability. In Figure \ref{clip_ablation}, a clip factor of $0.2$  performs nominally.
Varying the GAE factor does not introduce any significant changes overall. For Mini Batch size, we observe that both very small  sizes (i.e. 256), and very large sizes (i.e. 2048) introduce instability in the training process. Finally, lower discount factors exhibit faster learning rates, however, they show signs of catastrophic forgetting. A discount factor of $\gamma=0.9$ achieves the highest average reward  compared to its counterparts, but shows signs of instability during the training. The average cumulative reward returns here are computed over a rolling window of 1000 episodes.   
\begin{figure}[H]
\centering
\subfigure[Impact of Clipping Factor on Training Performances]{
\includegraphics[width=0.48\textwidth]{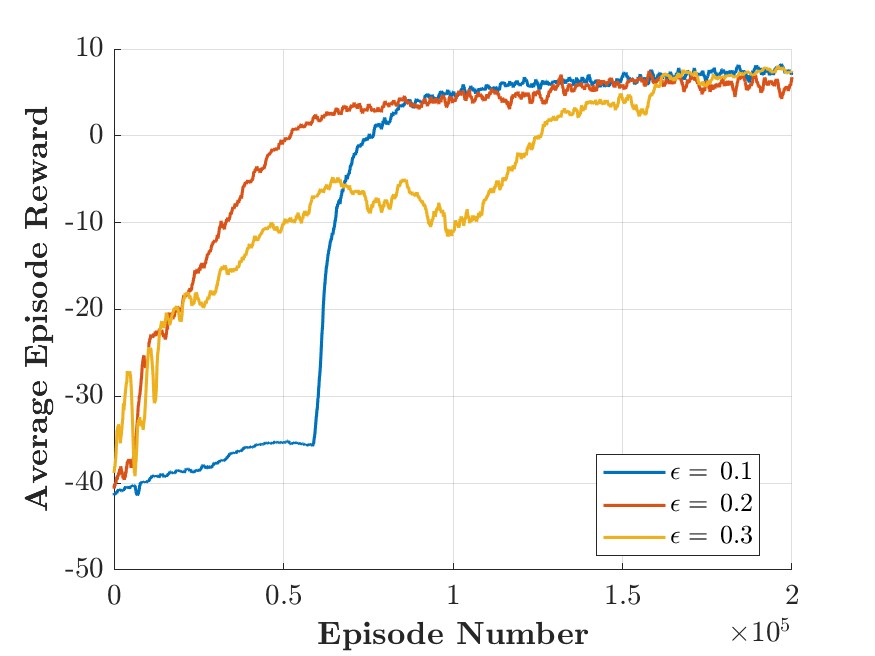}
\label{clip_ablation}}
\hfill
\subfigure[Effect of Actor and Critic Learning Rates on PPO Performance]{
\includegraphics[width=0.48\textwidth]{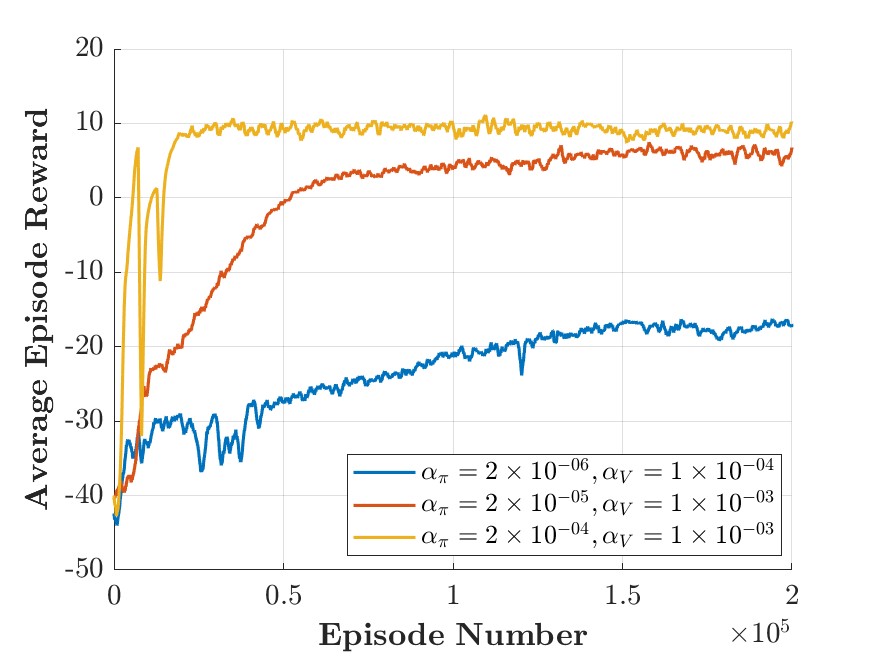} 
\label{rates_ablation}}
    
\vspace{0.1cm}
    
\subfigure[Effect of Discount Factor on Reward Dynamics]{
\includegraphics[width=0.48\textwidth]{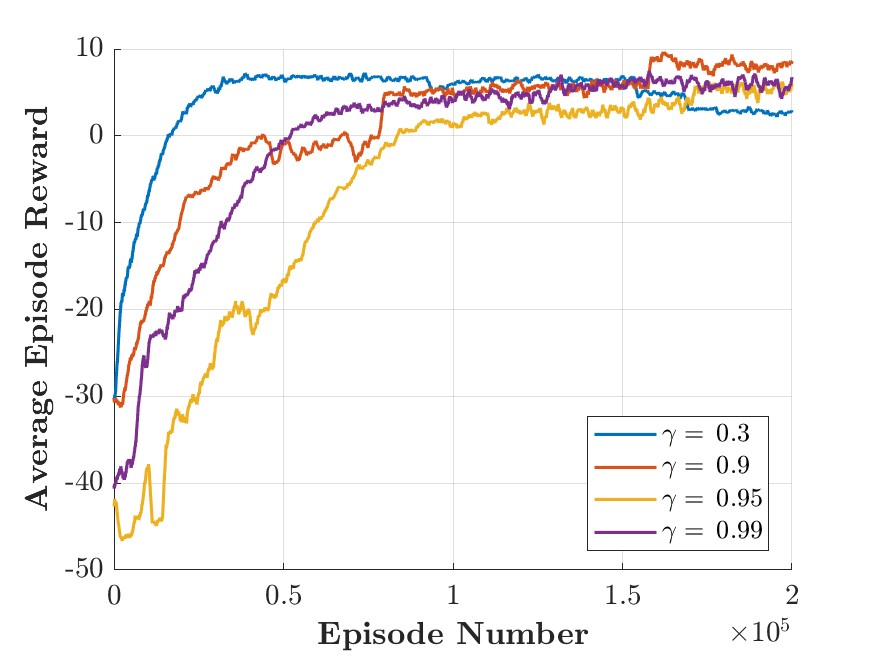}
\label{gamma_ablation}}
\hfill
\subfigure[Impact of Entropy Coefficients on PPO Training Stability and Learning Efficiency]{
\includegraphics[width=0.48\textwidth]{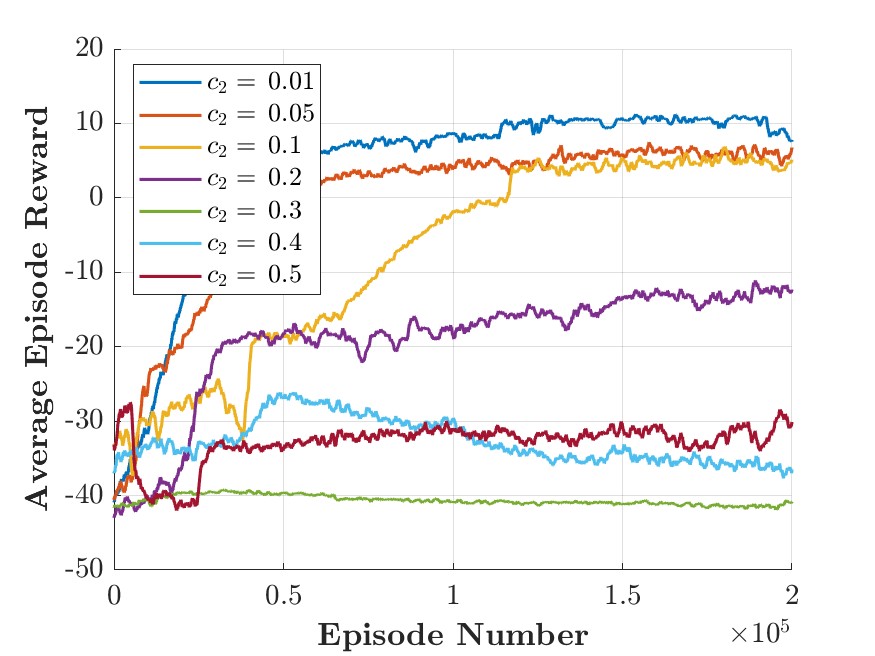} 
\label{entropy_ablation}}
    
\vspace{0.1cm}
    
\subfigure[Generalized Advantage Estimation (GAE) Variations]{
\includegraphics[width=0.48\textwidth]{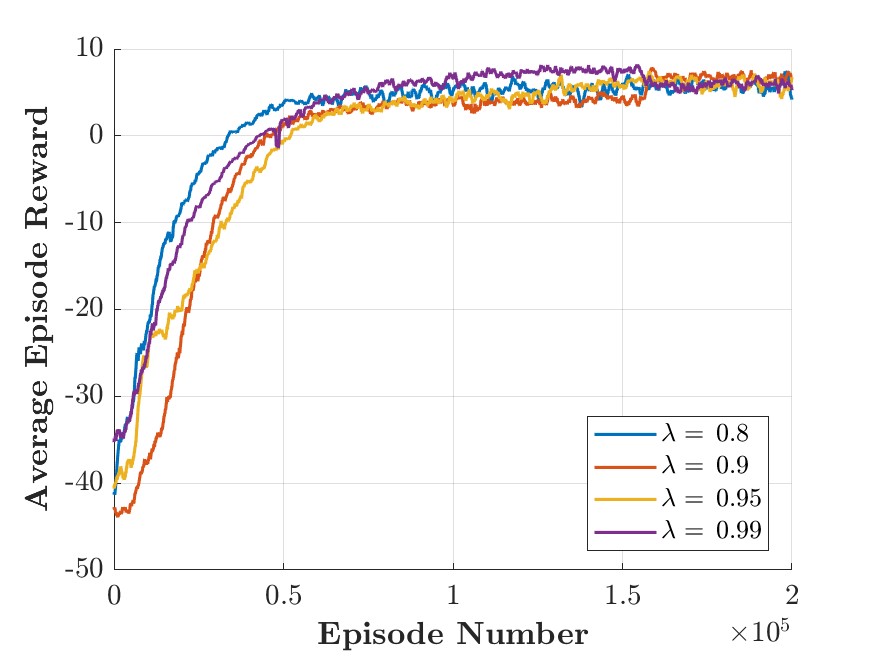}
\label{gae_ablation}}
\hfill
\subfigure[Impact of Mini-Batch Size on PPO Training Efficiency]{
\includegraphics[width=0.48\textwidth]{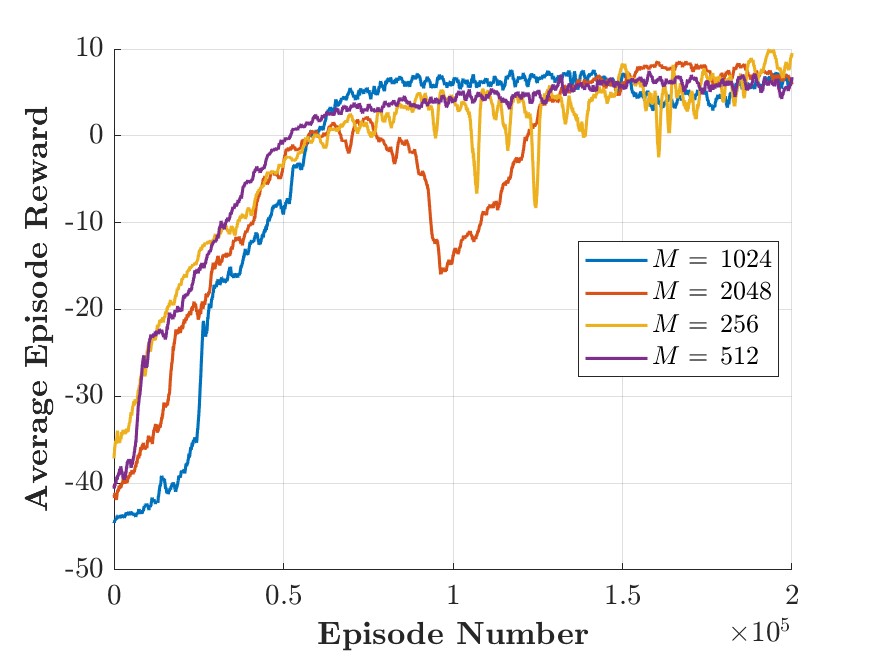}
\label{batch_ablation}}
    
\caption{Ablation Studies Showing the Impact of Different Hyperparameters on PPO\ Training.}
\label{fig:ablation_study}
\end{figure}

\section{Experiment and Results}\label{sec:experiments}

\subsection{Zero-Shot Learning Demonstration}
In this section, we demonstrate the applicability of the proposed adaptive grid refinement algorithm to higher dimensions and different domains by applying the agent trained on Problem~1 to a three-dimensional infinite-time horizon optimal control problem. Consider the problem of spacecraft detumbling, or stabilization of the angular rates of a rigid body.  The spacecraft dynamics are governed by the Euler rigid body equations, augmented by a control forcing term.
\begin{equation*}
\text{(P2)} \quad
\begin{aligned}
&\text{minimize } \quad J(u)=\frac{1}{2}\int_{t_0}^{\infty} \left[\omega^TQ\omega+u^TRu\right]\,dt,\\
&\text{subject to} \quad I \dot{\omega} = S(\omega) I \omega + b u, \\
&\text{where} \quad S(\omega) =
\begin{bmatrix}
0 & \omega_3 & -\omega_2 \\
-\omega_3 & 0 & \omega_1 \\
\omega_2 & -\omega_1 & 0
\end{bmatrix}.
\end{aligned}
\end{equation*}
Here, \( \omega \in \mathbb{R}^3 \) is the angular velocity of the spacecraft in a body-fixed reference frame, \( u \in \mathbb{R}^3 \) is the control torque vector, \( I = \text{diag}(1, 1.2, 0.8) \) is the principal moment of inertia matrix, and \( b = 1 \) is a constant scaling the applied control. The initial condition is \( \omega(t_0) = \omega_0 =[1, -0.8, 0.5] \) on $\Omega=[-3, 3]^3$, and the time range is \( t \in [t_0, \infty) \) with $h=0.01$. The running-cost weights are $Q=R=I_{3\times3}$.
 
In the test episode, the agent obtained a cumulative reward of 2.7774 with a breakdown of $\sum_{\kappa} R_E=0.9573$ and $\sum_{\kappa} R_N=1.9581$. The estimated integrated residual norm over $\Omega$ was $515.16$ which is lower compared to its initial value of $544.57$, and the number of nodes were lowered from $1502$ to $ 1170 $. On a subsequent episode, the estimated integrated residual norm was lowered from $535.4$ to $526.6$, while the number of nodes was reduced from $n_0=1502$ to $n_{\kappa_{\text{max}}}=1187$. Figures \ref{euler_grid} \& \ref{euler_traj} show, respectively, the modified grid and the trajectories generated on this grid for the first test episode. 
\begin{figure}[htbp]
\centering
\subfigure[Adaptively Modified Grid in $\Omega\subset\mathbb{R}^3$]{
\includegraphics[width=0.455\textwidth]{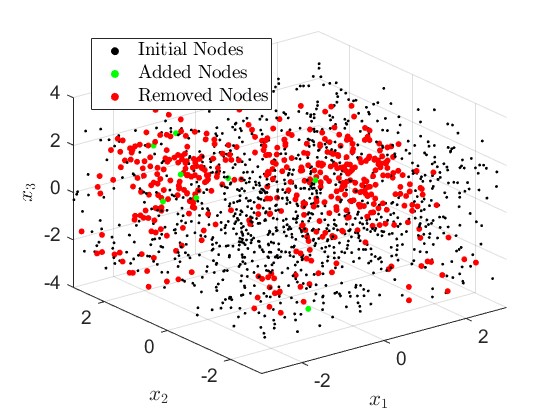}
\label{euler_grid}}
\hspace{0.02\textwidth}
\subfigure[Resulting State (Angular Rate) and Control (Torque) Trajectories on the Modified Grid]{
\includegraphics[width=0.455\textwidth]{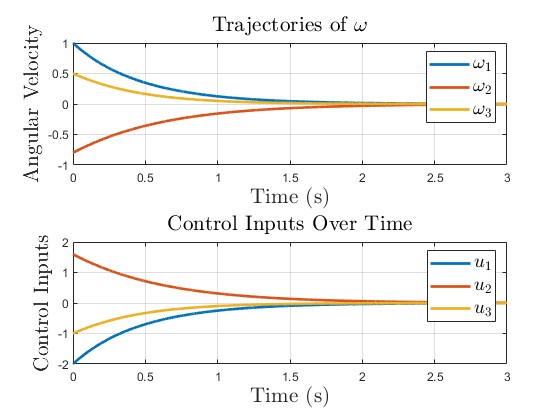}
 \label{euler_traj}}
    
\caption{Solution Trajectories of the Euler Stabilization Problem}
\label{fig:euler_grids}
\end{figure}

The next problem is an example of a 6-dimensional undiscounted infinite horizon control problem, augmenting the Euler dynamics with the three-state kinematic parameterization developed by Tsiotras \& Longuski\cite{tsiotras1995} to describe the orientation of the spacecraft. The control input applying torque to the spacecraft indirectly impacts the additional kinematic states $w_1$, $w_2$, and $z$ via its direct forcing of the angular velocity vector, $\omega$, as outlined in the Euler dynamics (P2).
\begin{equation*}
\text{(P3)} \quad
\begin{aligned}
I \dot{\omega} &= S(\omega) I \omega + b u, \\
\begin{bmatrix}\dot{w}_1 \\ \dot{w}_2 \\ \dot{z} \end{bmatrix} &= 
\begin{bmatrix}
\frac{1}{2} + \frac{1}{2}(w_1^2-w_2^2) & w_1w_2 & w_2 \\ 
w_1w_2 & \frac{1}{2} - \frac{1}{2}(w_1^2-w_2^2) & -w_1 \\ 
-w_2 & w_1 & 1
\end{bmatrix} \omega.
\end{aligned}
\end{equation*}
For this problem, the running cost weights were $Q=I_{6\times6}$ and $R=10I_{3\times3}$. Cost accumulation stops upon reaching the target, with zero terminal cost. The trajectories of this problem are shown below in Figure \ref{attitude}, where the feedback control is computed on the modified grid in $\Omega=[-1, 1]^6$.
The solution converged with $h=0.01$ on the initial grid, the terminal state was fixed at $(0, 0, 0, 0.5, 0.5, 0.5)$ as a boundary condition with $V=0$, and the initial state was chosen as $(0.05, -0.03, 0.10, 1, 0.5, 0.5)$. The node population was lowered from $n_0=5002$ to $n_{\kappa_{\text{max}}}=1343$, and $\mathcal{E}$ remained close to its initial value of $\mathcal{E}_0=10.30$ at 10.26.  
\begin{figure}[H]
\centering\includegraphics[width=3.7in]{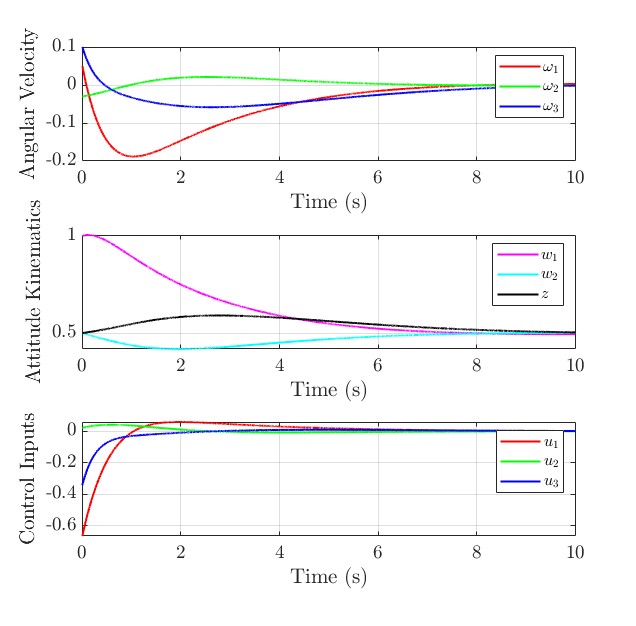}
\caption{Spacecraft Attitude Trajectories on the Modified Grid}
\label{attitude}
\end{figure}
The spacecraft regulation examples of this section use undiscounted
Bellman updates and assess empirical transfer beyond
the positive-discount setting of the preceding analysis.

\subsection{Applications of Grid Refinement to Fixed-Point Iterative Methods}
We now demonstrate the application of adaptive grid refinement to  approaches used for solving optimal control problems on grids, e.g. feedback optimal control problems in (Ref.~\citenum{chilan2020}). This application, outlined in Algorithm \ref{adaptive_refinement_alg}, is used iteratively within the fixed-point solver to adjust the grid structure while the solver is active. Specifically, the agent is allowed to modify the grid until a specified iteration stopping tolerance, $\epsilon_V$, is reached. At scheduled solver iterations, an intermediate refinement  episode is performed when $\mathcal{E}>\tau$. If $\mathcal{E}\le\tau$ at a scheduled refinement, a final episode is performed, after which the grid is held fixed. Each episode terminates when its cumulative reward becomes positive, the maximum number of actions is reached, or the environment signals termination. Other variations of this algorithm may be feasible, depending on the application.

\begin{algorithm}[H]
\caption{Integrating Adaptive Grid Refinement in HJB Solver}
\label{adaptive_refinement_alg}
\begin{algorithmic}[1]
\Require Initial grid $\mathcal{X}_0$, value function $V_0$, control function $u_0$, agent $\pi_\theta$, maximum episode steps $\kappa_{\max}\ge1$, residual threshold $\tau$, stopping tolerance $\epsilon_V>0$, initial refinement interval $d_{\text{iter}}\ge1$
\Ensure Approximate value function $V$ and control $u$ on grid $\mathcal{X}$, satisfying the iteration stopping criterion

\State $\mathcal{X}\gets\mathcal{X}_0$, $V\gets V_0$, $u\gets u_0$
\State $k\gets0$, $\delta\gets+\infty$, $\text{grid static}\gets\text{False}$

\While{$\delta>\epsilon_V$}
\State \textbf{Fixed-point update:}
\State $V_{\text{old}}\gets V$
\State Compute one Bellman update of $V$ and $u$ on $\mathcal{X}$
\State $\delta\gets\max_{x_i\in\mathcal{X}}|V(x_i)-V_{\text{old}}(x_i)|$
\State $k\gets k+1$
\State Compute the estimated residual norm $\mathcal{E}$

\If{$\neg\text{grid static}\text{ and }k\ge d_{\text{iter}}$}
\State \textbf{Refinement episode:}
\State $\text{final episode}\gets(\mathcal{E}\le\tau)$
\State $s\gets\text{Reset}(\mathcal{X},V,u)$
\State $\kappa\gets0$, $\mathcal{R}_{\text{cumulative}}\gets0$, $\text{done}\gets\text{False}$

\While{$\kappa<\kappa_{\max}\text{ and }\mathcal{R}_{\text{cumulative}}\le0\text{ and }\neg\text{done}$}
\State Sample action $a\sim\pi_\theta(\cdot\mid s)$
\State $(s,\mathcal{R},\text{done})\gets\text{Step}(a)$
\State $\mathcal{R}_{\text{cumulative}}\gets\mathcal{R}_{\text{cumulative}}+\mathcal{R}$
\State $\kappa\gets\kappa+1$
\EndWhile

\State Update $\mathcal{X}$, $V$, and $u$ from the refinement environment
\If{$\text{final episode}$}
\State $\text{grid static}\gets\text{True}$
\EndIf
\State $\delta\gets+\infty$
\State $k\gets0$
\State $d_{\text{iter}}\gets2d_{\text{iter}}$
\EndIf
\EndWhile

\State \Return $(V,u,\mathcal{X})$
\end{algorithmic}
\end{algorithm}

To demonstrate, we implemented Algorithm \ref{adaptive_refinement_alg} to solve Problem 1, starting from a naive guess.
In particular, the solver with the integrated grid refinement algorithm met the iteration stopping criterion after 2543 iterations with a tolerance of $\epsilon_V=1\times10^{-14}$ and the final estimated residual norm was 4.4523 with 487 nodes. The same problem was solved on the static grid as shown in Figure~\ref{initial_grid}. In the static grid case, the solver met the iteration stopping criterion after 2639 iterations with the same tolerance and the final estimated residual norm was 6.8150 with 1502 nodes. The solver with the integrated adaptive grid converged 258.2340 seconds earlier.

\section{Conclusion}
This paper has introduced a reinforcement learning-based adaptive grid refinement method for solving the HJB PDE. By balancing the reduction of node population with the reduction of the residual norm, our approach has demonstrated its ability to optimize computational resources while reducing the estimated residual norm. Notably, the algorithm's capability to dynamically add and remove nodes within a defined domain resulted in reductions in the estimated integrated residual norm across the two-dimensional test cases. We suggest that randomizing the initial node count during training could further improve the node placement and removal process while increasing robustness. 

Future work will extend the proposed adaptive grid-refinement method toward reliable transfer across state dimensions by developing dimension-agnostic observations based on normalized local spacing, occupancy, and summaries of neighborhood geometry and interpolation quality. One direction for future work is to explore the use of a common scrambled Sobol query set within each refinement episode, together with recomputing greedy controls at the evaluation points, to assess whether these choices improve the reliability of the residual-based reward as an indicator of refinement quality. These changes may reduce comparison noise and measure consistency with the minimized HJB equation. To strengthen the connection between the reward and approximation accuracy, we will investigate residual-to-error estimates accounting for domain coverage, boundary error, and discretization consistency. Refinement acceptance criteria will additionally control
residual growth, prevent grid degeneracy, and maintain
admissible footpoints and monotone interpolation. In
settings admitting a strict contraction estimate, these
criteria will also preserve the contraction property of
each fixed-grid Bellman operator. The analysis will also characterize when changes in grid geometry improve the reconstructed value function. Specifically, let $V^\star$ denote the exact value function on $\Omega$, and let 
$\widehat V_X$ and $\widehat V_Y$ denote its reconstructions on grids $X$ and $Y$. Defining \(e_X=V^\star-\widehat V_X\) and \(\delta=\widehat V_Y-\widehat V_X\), with \(e_X,\delta\in L^2(\Omega)\), the true $L^2(\Omega)$ approximation error decreases precisely when
\[
2\langle e_X,\delta\rangle_{L^2(\Omega)}
>
\|\delta\|_{L^2(\Omega)}^2.
\]
This condition requires sufficient alignment between the interpolation correction and the existing error without excessive overshoot. Since the exact error is generally unavailable during training, this characterization will guide the theoretical analysis, while computable residual estimates and stability safeguards will guide refinement decisions.

The methodology presented in this paper is not limited to HJB equations and can be extended to other types of PDEs. A key highlight of our work is the demonstration of zero-shot generalization in grid refinement for problems in higher dimensions. The results from higher-dimensional demonstrations mark a promising step toward showcasing limited scalability of the algorithm.

\newpage
\section{Acknowledgment}
The authors acknowledge the support of the U.S. Air Force Research Laboratory (AFRL) under grant number FA8651-23-1-0001. Any opinions and findings in this paper are of the authors and do not reflect those of the AFRL. This work used the Accelerating Computing for Emerging Sciences (ACES) at Texas A\&M University through allocation MTH240017 from the Advanced Cyberinfrastructure Coordination Ecosystem: Services \& Support (ACCESS) program, which is supported by U.S. National Science Foundation grant numbers 2138259, 2138286, 2138307, 2137603, and 2138296\cite{boerner2023}.

\appendix
\section*{Appendix: Theorem 1}
We have\begin{equation*}
\|\,V_{h}-V_{h}^{k}\|_{\infty}\;\le\;
\frac{C}{\lambda\,h}
\bigl\|V_{h}-\mathcal{I}_{k}[V_{h}]\bigr\|_{\infty}.
\end{equation*}

\begin{proof}
For each node $x_{i}\in\mathcal{X}$, we define the nodal difference
\begin{equation*}
D_{i}\equiv V_{h}(x_{i}) - V_{h}^{k}(x_{i}),
\end{equation*}
and for each $x\in\Omega$, the interpolation error
\begin{equation*}
E_x\equiv
V_{h}(x) - \mathcal{I}_{k}[V_{h}](x).
\end{equation*}
Previously we have reviewed that for $\lambda>L_{f}$, Equations \eqref{hjb_h} and \eqref{hjb_h_x} have unique solutions $V_h(x)$, $x\in\Omega$ and $V^k_h(x_{i})$, $x_i\in\mathcal{X}$ respectively. If $x_i+hf(x_i,a)\in\Omega$ is not located at a node, we interpolate $V_{h}^{k}$ from its values $\{V_{h}^{k}(x_{j})\}$ using (convex combination) interpolation weights.

Let $u_i^*\in U$ achieve the minimum in Eq. \eqref{hjb_h_x}. Then we have that
\begin{equation*}
V_{h}^{k}(x_{i})=(1-\lambda\,h)\sum_{j=1}^{n}\lambda_{ij}(u^*_{i})V_{h}^{k}(x_{j})+h\,r(x_{i},u^*_{i}).
\end{equation*}
But from \eqref{hjb_h}, the same control yields,
\begin{equation*}
V_{h}(x_{i})\le(1-\lambda\,h)\;V_{h}\bigl(x_{i}+h\,f(x_{i},u^*_{i})\bigr)+h\,r(x_{i},u^*_{i}).
\end{equation*}
Subtracting, we get
\begin{equation*}
D_{i}\le(1-\lambda\,h)\Bigl[V_{h}(x_{i}+h\,f(x_{i},u^*_{i}))-\sum_{j}\lambda_{ij}(u^*_{i})\,V_{h}^{k}(x_{j})
\Bigr].
\end{equation*}
Now
\begin{equation*}
V_{h}(x_{i}+h\,f(x_{i},u^*_{i}))=\sum_{j}\lambda_{ij}(u^*_{i})V_{h}(x_{j})+[V_{h}(\xi)-\mathcal{I}_{k}[V_{h}](\xi)],\quad \xi = x_{i}+h\,f(x_{i},u^*_{i}).
\end{equation*}
Hence
\begin{equation*}
D_{i}
\le(1-\lambda\,h)\sum_{j}\lambda_{ij}(u^*_{i})\bigl[V_{h}(x_{j})-V_{h}^{k}(x_{j})\bigr]+(1-\lambda\,h)\;E_{\xi}.
\end{equation*}
This shows
\begin{equation*}
D_{i}\le(1-\lambda\,h)\,\max_{j}\bigl|D_{j}\bigr|+(1-\lambda h)\|E\|_{\infty}.
\end{equation*}
Similarly, let $u_i^\circ\in U$ achieve the minimum in Eq. \eqref{hjb_h} at $x_i$.  Then
\begin{equation*}
V_{h}(x_{i})=(1-\lambda h) V_{h}(x_{i}+h f(x_{i},u^{\circ}_{i}))+h\,r(x_{i},u^{\circ}_{i}),
\end{equation*}
while
\begin{equation*}
V_{h}^{k}(x_{i})\le(1-\lambda\,h)\sum_{j}\lambda_{ij}(u^{\circ}_{i})\,V_{h}^{k}(x_{j})+h\,r(x_{i},u^{\circ}_{i}).
\end{equation*}
Setting $\xi=x_i+h f(x_i,u_i^\circ)$ and subtracting the two, we have
\begin{equation*}
D_{i}=V_{h}(x_{i})-V_{h}^{k}(x_{i})\ge(1-\lambda\,h)\Bigl[\sum_{j}\lambda_{ij}(u^{\circ}_{i})\,V_{h}(x_{j})-\sum_{j}\lambda_{ij}(u^{\circ}_{i})\,V_{h}^{k}(x_{j})\Bigr]+(1-\lambda\,h)\,E_{\xi}.
\end{equation*}
Thus
\begin{equation*}
D_{i}\ge-(1-\lambda\,h)\,\max_{j}\bigl|D_{j}\bigr|-(1-\lambda\,h)\,\|E\|_{\infty}.
\end{equation*}

Combining the upper and lower bounds shows
\begin{equation*}
|D_{i}|
\le(1-\lambda\,h)\,\|D\|_{\infty}+(1-\lambda\,h)\,\|E\|_{\infty}\quad\Longrightarrow\quad\|D\|_{\infty}\le(1-\lambda\,h)\,\|D\|_{\infty}+(1-\lambda\,h)\,\|E\|_{\infty}.
\end{equation*}
Therefore,
\begin{equation*}
\|D\|_{\infty}=\max_{i}\bigl|V_{h}(x_{i})-V_{h}^{k}(x_{i})\bigr|\le\frac{1-\lambda\,h}{\lambda\,h}\,\|E\|_{\infty}.
\end{equation*}

Next, let  $x\in\Omega$ be an interior point and $\{\lambda_{j}(x)\}$ be its interpolation coordinates. Then
\begin{equation*}
V_{h}^{k}(x)=\sum_{j}\lambda_{j}(x)\;V_{h}^{k}(x_{j}),
\end{equation*}
so
\begin{equation*}
V_{h}(x) - V_{h}^{k}(x)=\Bigl[V_{h}(x)-\sum_{j}\lambda_{j}(x)\,V_{h}(x_{j})\Bigr]+\sum_{j}\lambda_{j}(x)\,\bigl[V_{h}(x_{j})-V_{h}^{k}(x_{j})\bigr].
\end{equation*}
Taking the supremum in $x$,
\begin{equation*}
\|V_{h}-V_{h}^{k}\|_{\infty}\le\|E\|_{\infty} + \|D\|_{\infty}\le\|E\|_{\infty}+\frac{1-\lambda\,h}{\lambda\,h}\;\|E\|_{\infty}\le\frac{1}{\lambda\,h}\;\|E\|_{\infty},
\end{equation*}
whenever $0<\lambda h\le1$. 
\end{proof}

\bibliographystyle{AAS_publication} 
\bibliography{references} 

\end{document}